\documentclass[preprint,3p]{elsarticle}

\date{\today}

\usepackage{mathrsfs}
\usepackage{amsfonts, amssymb}
\usepackage{mathtools}
\usepackage{amsthm}
\usepackage{txfonts}
\usepackage{lmodern}
\usepackage{multirow}
\usepackage{tikz}
\usetikzlibrary{arrows.meta,decorations.pathreplacing,decorations.markings,shapes,calc}
\usetikzlibrary{matrix,arrows,decorations.pathmorphing,backgrounds,decorations.markings,positioning}
\tikzset{2cell/.style={-implies,double,double equal sign distance,shorten 
>=2pt, shorten <=3pt}}
\tikzset{2cellshort/.style={-implies,double,double equal sign distance,shorten 
>=4pt, shorten <=5pt}}
\tikzset{2cellr/.style={implies-,double,double equal sign distance,shorten 
>=3pt, shorten <=2pt}}
\tikzset{3cell/.style={-implies,double,double distance=2.5pt,shorten >=2pt, 
shorten <=3pt}}
\tikzset{labelsize/.style={font=\scriptsize}}
\tikzset{string/.style={very thick}}
\tikzset{
  pto/.style={->,postaction={decorate},
    decoration={
        markings,
        mark=at position 0.5 with {\arrow{|}}}
  },
}

\usepackage[colorlinks=true,pagebackref=true]{hyperref}

\newtheorem{theorem}{Theorem}[section]
\newtheorem{lemma}[theorem]{Lemma}
\newtheorem{proposition}[theorem]{Proposition}
\newtheorem{corollary}[theorem]{Corollary}

\newdefinition{definition}[theorem]{Definition}
\newdefinition{example}[theorem]{Example}
\newdefinition{remark}[theorem]{Remark}

\makeatletter
\newcommand{\dotminus}{\mathbin{\text{\@dotminus}}}

\newcommand{\@dotminus}{%
  \ooalign{\hidewidth\raise1ex\hbox{.}\hidewidth\cr$\m@th-$\cr}%
}
\makeatother

\mathchardef\mhyphen="2D

\makeatletter
\newcommand{\pto}{}
\newcommand{\pgets}{}
\DeclareRobustCommand{\pto}{\mathrel{\mathpalette\p@to@gets\to}}
\DeclareRobustCommand{\pgets}{\mathrel{\mathpalette\p@to@gets\gets}}
\newcommand{\p@to@gets}[2]{%
  \ooalign{\hidewidth$\m@th#1\mapstochar\mkern5mu$\hidewidth\cr$\m@th#1\longrightarrow$\cr}%
}
\makeatother

\newcommand{\id}[1]{{\mathrm{id}_{#1}}}
\newcommand{\ob}[1]{{\mathrm{ob}(#1)}}

\newcommand{\op}{{\mathrm{op}}}

\newcommand{\cat}[1]{{\mathcal{#1}}}

\newcommand{\Lan}{\mathrm{Lan}}
\newcommand{\Ran}{\mathrm{Ran}}

\newcommand{\enCat}[1]{{{#1}\mhyphen\mathbf{Cat}}}

\newcommand{\defemph}[1]{\textbf{#1}}

\newcommand{\one}{I}
\newcommand{\tms}{\circ}

\newcommand{\mto}{\mathbin{\searrow}}
\newcommand{\mot}{\mathbin{\swarrow}}

\newcommand{\nonnegRmin}{{\overline{\mathbb{R}}^{\geq 0}_+}}

\newcommand{\nonnegRminmax}{{\overline{\mathbb{R}}^{\geq 0}_{\mathrm{max}}}}
\newcommand{\TV}{{\mathbf{2}}}

\newcommand{\UA}{\mathcal{U}}

\newcommand{\Pow}{\mathcal{P}}

\newcommand{\Isb}{\MN}

\newcommand{\MN}{\mathcal{M}} 

\newcommand{\psh}{\mathcal{P}}
\newcommand{\copsh}{\mathcal{P}^\dagger}

\begin{document}
	
\begin{frontmatter}
	\title {Completeness and injectivity}
	
	\author{Soichiro Fujii\corref{cor}\fnref{fn1}}
	\ead{s.fujii.math@gmail.com}
	
	\cortext[cor]{Corresponding author.}
	\fntext[fn1]{The author is supported by ERATO HASUO Metamathematics for Systems Design Project (No. JPMJER1603), JST.}
	
	\address{Research Institute for Mathematical Sciences, Kyoto University,
		Kyoto 606-8502, Japan}
	
\begin{abstract}
We show that for any quantale $\mathcal{Q}$, a $\mathcal{Q}$-category is skeletal and complete if and only if it is injective with respect to fully faithful $\mathcal{Q}$-functors. This is a special case of known theorems due to Hofmann and Stubbe, but we provide a different proof, using the characterisation of the MacNeille completion of a $\mathcal{Q}$-category as its injective envelope.
For Lawvere metric spaces, our results yield those of Kemajou, K\"unzi and Otafudu. We point out that their notion of Isbell convexity can be seen as a geometric formulation of categorical completeness for Lawvere metric spaces.
\end{abstract}


\end{frontmatter}

\section{Introduction}\label{sec:intro}
The main purpose of this paper is to present a theorem claiming the equivalence between \emph{completeness} and \emph{injectivity} in the context of \emph{quantale-enriched categories}. 
In order to convey the idea of the theorem, we start with describing two classical theorems in order theory and metric space theory, and then a variant of the latter for directed metric spaces; the first and third of them are instances of our theorem (modulo minor modifications).

\begin{theorem}[{\cite{Banaschewski_Bruns}}]\label{thm:BB}
	A poset is a complete lattice if and only if it is injective.
\end{theorem}
Recall that a poset $\cat{E}$ is a \defemph{complete lattice} if it has all suprema (or equivalently all infima) of subsets of $\cat{E}$. On the other hand, a poset $\cat{E}$ (whose partial order relation we denote by $\preceq_\cat{E}$) is said to be \defemph{injective} if, whenever we have posets $\cat{C}$ and $\cat{D}$, a monotone map $f\colon \cat{C}\longrightarrow\cat{E}$ (i.e., a function such that for all $c,c'\in\cat{C}$, $c\preceq_\cat{C} c'$ implies $f(c) \preceq_\cat{E} f(c')$), and an order embedding $i\colon\cat{C}\longrightarrow \cat{D}$ (i.e., a function such that $c\preceq_\cat{C} c'$ if and only if $i(c) \preceq_\cat{D} i(c')$), there exists a (not necessarily unique) monotone map $g\colon\cat{D}\longrightarrow\cat{E}$ making the diagram
\[
\begin{tikzpicture}[baseline=-\the\dimexpr\fontdimen22\textfont2\relax ]
\node (TL) at (0,1) {$\cat{C}$};
\node (BL) at (0,-1) {$\cat{D}$};
\node (TR) at (2,1) {$\cat{E}$};

\draw [->] (TL) to node [auto,labelsize] {$f$} (TR);
\draw [>->] (TL) to node [auto,labelsize,swap] {$i$} (BL);
\draw [->,dashed] (BL) to node [auto,labelsize,swap] {$g$} (TR);
\end{tikzpicture}
\]
commute.
Notice that the property of being a complete lattice is described totally in terms of the \emph{internal} structure of a poset, whereas that of being injective is formulated solely in terms of its \emph{external} behaviour among all posets.
The fascination of this theorem due to Banaschewski and Bruns lies in the fact that it connects an internal property with an external one.

\begin{theorem}[{\cite{Aronszajn_Panitchpakdi_extension}}]\label{thm:AP}
	A metric space is hyperconvex if and only if it is injective.
\end{theorem}

A metric space $\cat{E}$ (whose distance function we denote by $d_\cat{E}$) is \defemph{hyperconvex} if, for any (possibly infinite) family $((e_i,r_i))_{i\in I}$ of pairs of a point $e_i\in\cat{E}$ and a nonnegative real number $r_i$ satisfying $r_i+r_j\geq d_\cat{E}(e_i,e_j)$ for all $i,j\in I$, there exists a point $e\in\cat{E}$ such that $r_i\geq d_\cat{E}(e_i,e)$ for all $i\in I$.
Let us elaborate the definition. One can view each pair $(e_i,r_i)$ as the closed ball $B(e_i,r_i)$ in $\cat{E}$ with centre $e_i$ and radius $r_i$. Then the condition that $r_i+r_j\geq d_\cat{E}(e_i,e_j)$ for all $i,j\in I$ says that any two balls in the family \emph{potentially intersect}; indeed, if $r_i+r_j < d_\cat{E}(e_i,e_j)$ then $B(e_i,r_i)\cap B(e_j,r_j)=\emptyset$ by the triangle inequality.
The existence of a point $e\in\cat{E}$ such that $r_i\geq d_\cat{E}(e_i,e)$ for all $i\in I$ means that the intersection $\bigcap_{i\in I}B(e_i,r_i)$ of \emph{all} balls in the family is nonempty.
For example, $\mathbb{R}^2$ with the Euclidean metric is not hyperconvex, but $\mathbb{R}^2$ with the maximum metric is; the following are some balls in these metric spaces.
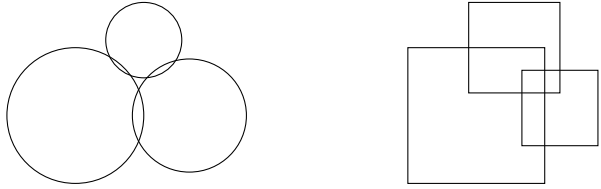
\[
\begin{tikzpicture}[baseline=-\the\dimexpr\fontdimen22\textfont2\relax ]
\draw (0,0.1) circle (0.9cm);
\draw (1.5,0.1) circle (0.75cm);
\draw (0.9,1.1) circle (0.5cm);
\end{tikzpicture}
\qquad\qquad\qquad
\begin{tikzpicture}[baseline=-\the\dimexpr\fontdimen22\textfont2\relax ]
\draw (-0.8,-0.8) rectangle (1,1);
\draw (0.7,0.7) rectangle (1.7,-0.3);
\draw (0,0.4) rectangle (1.2,1.6);
\end{tikzpicture}
\]

The definition of injectivity for metric spaces parallels that for posets. A metric space $\cat{E}$ is \defemph{injective} if, whenever we have metric spaces $\cat{C}$ and $\cat{D}$, a nonexpansive map $f\colon \cat{C}\longrightarrow\cat{E}$ (i.e., a function such that $d_\cat{C}(c,c')\geq d_\cat{E}(f(c), f(c'))$), and an isometric embedding $i\colon\cat{C}\longrightarrow \cat{D}$ (i.e., a function such that $d_\cat{C}(c, c')=d_\cat{D}(i(c), i(c'))$), there exists a (not necessarily unique) nonexpansive map $g\colon\cat{D}\longrightarrow\cat{E}$ such that $f=g\circ i$.
Again, this theorem of Aronszajn and Panitchpakdi is interesting in that it relates the internal property of hyperconvexity with the external one of injectivity.

The following theorem due to Kemajou, K\"unzi and Otafudu is the directed variant of Theorem \ref{thm:AP}. By a \defemph{di-space} we mean a possibly nonsymmetric (in the sense that $d_\cat{C}(c,c')$ may be different from $d_\cat{C}(c',c)$) generalisation of a metric space.
\begin{theorem}[{\cite{Kemajou_Kunzi_Otafudu_Isbell_hull}}]\label{thm:KKO}
	A di-space is Isbell convex if and only if it is injective.
\end{theorem}
The notion of Isbell convexity is a straightforward adaptation of hyperconvexity to the nonsymmetric setting. Precisely, a di-space $\cat{E}$ is \defemph{Isbell convex} if, for any (possibly infinite) family $((e_i,x_i,y_i))_{i\in I}$ of \emph{triples} of a point $e_i\in\cat{E}$ and nonnegative real numbers $x_i$ and $y_i$ satisfying $x_i+y_j\geq d_\cat{E}(e_i,e_j)$ for all $i,j\in I$, there exists a point $e\in\cat{E}$ such that $x_i\geq d_\cat{E}(e_i,e)$ and $y_i\geq d_\cat{E}(e,e_i)$ for all $i\in I$. 
We shall say more about this notion in Section \ref{sec:Isb_conv}. The definition of injectivity for di-spaces is completely parallel to that for metric spaces.

\medskip

Except that a precise relationship between the notions of complete lattice on the one hand, and of hyperconvex metric space and Isbell convex di-space on the other, is perhaps not apparent, Theorems \ref{thm:BB}--\ref{thm:KKO} look quite similar.
These results look even closer if one notes the fact that in all cases we have constructions of \emph{injective envelopes}.
Informally, an injective envelope of an object (i.e., poset, metric space or di-space) $\cat{C}$ is the smallest injective object $\overline{\cat{C}}$ to which $\cat{C}$ embeds; we shall give a precise definition in an abstract setting in Section~\ref{sec:formal_theory}.
The injective envelope of a poset, metric space and di-space is also known as its \emph{MacNeille completion} \cite{MacNeille_poset}, \emph{hyperconvex hull} or \emph{tight span} \cite{Isbell_six,Dress_tight_ext,Herrlich_hyperconvex}, and \emph{Isbell hull} or \emph{directed tight span} \cite{Kemajou_Kunzi_Otafudu_Isbell_hull,Hirai_Koichi_tight_span}, respectively.

\medskip

In this paper we shall prove a generalisation of Theorems \ref{thm:BB} and \ref{thm:KKO}. (For a generalisation of Theorem~\ref{thm:AP}, see \cite{Jawhari_Misane_Pouzet_retracts}.)  We unify posets and di-spaces by \emph{categories enriched over a quantale}. A quantale \cite{Mulvey_and} $\cat{Q}$ is a complete lattice $(Q,\preceq_{\cat{Q}})$ equipped with a compatible monoid structure $(Q,I_\cat{Q},\tms_{\cat{Q}})$. Given any quantale $\cat{Q}$, one can consider categories enriched over $\cat{Q}$, or \emph{$\cat{Q}$-categories} \cite{Kelly:enriched,Stubbe_quantaloid_dist,Stubbe_tensor_cotensor}.  Informally, a $\cat{Q}$-category $\cat{C}$ is a set $\ob{\cat{C}}$ equipped with a $\cat{Q}$-valued preorder relation $\cat{C}(-,-)\colon\ob{\cat{C}}\times\ob{\cat{C}}\longrightarrow Q$.
Taking $\cat{Q}=\TV$, the two-element quantale, we recover preordered sets as $\TV$-categories, whereas taking $\cat{Q}=\nonnegRmin$, the quantale of extended nonnegative real numbers with addition as the monoid structure, we obtain a mild generalisation of di-spaces (called \emph{Lawvere metric spaces}) as $\nonnegRmin$-categories \cite{Lawvere_metric}.

Theorems \ref{thm:BB} and \ref{thm:KKO} can be generalised as follows.
\begin{theorem}[{\cite{Hofmann_injective,Stubbe_note}}]
	\label{thm:main}
	Let $\cat{Q}$ be a quantale. A $\cat{Q}$-category is skeletal and complete if and only if it is injective (with respect to fully faithful $\cat{Q}$-functors).
\end{theorem}
The terms appearing in the above statement will be introduced in Section~\ref{sec:Q-cat}.
Actually, this theorem is known in much more general settings. In \cite[Theorem 2.7]{Hofmann_injective} it is proved for $\mathcal{T}$-categories for a topological theory $\mathcal{T}$, and in \cite{Stubbe_note} it is proved for categories enriched over a quantaloid, with attribution to Hofmann for private communication. 
See also \cite[Proposition 5.2]{Stubbe_double} and \cite[Theorem 10.1]{Shen_Tholen_topological}. 

The MacNeille completion can be generalised from posets to $\cat{Q}$-categories; see \cite[Definition 7.2]{Garner_topological} and \cite[Definition 5.5.2]{Shen_thesis}.
It is also known that the MacNeille completion for $\nonnegRmin$-categories coincides with the Isbell hull \cite{Willerton_tight_span}.
We show that the abstract characterisation of the MacNeille completion of a poset as its injective envelope \cite{Banaschewski_Bruns} also extends to $\cat{Q}$-categories.
\begin{theorem}
	\label{thm:main_envelope}
	Let $\cat{Q}$ be a quantale. For any $\cat{Q}$-category $\cat{C}$, its MacNeille completion $\Isb\cat{C}$ is the injective envelope of $\cat{C}$.
\end{theorem}
In fact, we shall prove Theorem \ref{thm:main} using Theorem \ref{thm:main_envelope}; our proof of Theorem~\ref{thm:main} extends the proof of Theorem \ref{thm:BB} in \cite{Banaschewski_Bruns}, and is different from those adopted in \cite{Hofmann_injective,Stubbe_note,Stubbe_double,Shen_Tholen_topological}. If one assumes enough background on enriched categories, the latter proofs are arguably shorter, but we believe that our proof illuminating the role of the MacNeille completion is of independent interest.

\medskip

The outline of this paper is as follows. 
In Sections \ref{sec:quantale} and \ref{sec:Q-cat} we introduce background materials on quantales and $\cat{Q}$-categories respectively.
In Section \ref{sec:Isbell} we explain the MacNeille completion for $\cat{Q}$-categories. Then, after a brief review of a formal theory of injective envelopes in Section \ref{sec:formal_theory}, we prove Theorems \ref{thm:main_envelope} and \ref{thm:main} in Section \ref{sec:complete_iff_injective}.
Finally, in Section~\ref{sec:Isb_conv}, we revisit the notion of Isbell convexity of \cite{Kemajou_Kunzi_Otafudu_Isbell_hull}, and point out that it is equivalent to categorical completeness.

\subsection*{Acknowledgements}
I am grateful to the organisers of the Third Pan-Pacific International Conference on Topology and Applications for giving me an opportunity to present this material. I would like to thank Lili Shen for funding my visit to Chengdu, and for providing valuable comments on this work, including the information on the works of Hofmann, Stubbe, Shen and Tholen.

\section{Quantales}\label{sec:quantale}
We first introduce \emph{quantales} \cite{Mulvey_and}, also known as \emph{complete 
	idempotent semirings} \cite{CGQ_duality,LMS_idem_fa}.
They are the enriching (or base) categories
in the portion of enriched category theory we shall be concerned with.

\begin{definition}\label{def:quantale}
	A (unital) \defemph{quantale} $\cat{Q}$
	is a complete lattice $(Q,\preceq_\cat{Q})$
	equipped with a monoid structure $(Q,\one_\cat{Q},\tms_\cat{Q})$
	such that the multiplication $\tms_\cat{Q}$ preserves arbitrary suprema in each variable:
	$(\bigvee_{i\in I} y_i)\tms_\cat{Q} x = \bigvee_{i\in I}(y_i\tms_\cat{Q} 
	x)$ and 
	$y\tms_\cat{Q} (\bigvee_{i\in I} x_i) = \bigvee_{i\in I}(y\tms_\cat{Q} x_i)$.
	We often omit the subscript $\cat{Q}$ from the data of a quantale,
	writing it simply as $\cat{Q}=(Q,\preceq,\one,\tms)$.
\end{definition}

Notice that in the definition of quantale,
we do {not} assume commutativity of the multiplication $\tms$ by default; 
quantales with commutative multiplication are said to be 
\defemph{commutative}.

\medskip
The notion of \emph{adjunction} is central to category theory. In this paper we shall only need the particularly simple case of adjunctions between posets, also known as \emph{Galois connections}.
Recall that given posets $(L,\preceq)$ and $(L',\preceq')$,
two functions $f\colon L\longrightarrow L'$
and $u\colon L'\longrightarrow L$ are said to form an \defemph{adjunction}
if, for any $l\in L$ and $l'\in L'$,
\begin{equation}\label{eqn:adjointness_poset}
f(l)\preceq' l' \iff l\preceq u(l')
\end{equation}
holds. 
We call $f$ the \defemph{left adjoint} of $u$ and $u$ the \defemph{right 
	adjoint} of $f$,
and write them as $f\dashv u$.
The adjointness relation (\ref{eqn:adjointness_poset}) is powerful enough to 
determine each of the functions $f$ and $u$ from the other, and force both of 
them to be \emph{monotone} functions \cite{Street_core}.

We record the following well-known fact.
\begin{proposition}\label{prop:adj_funct_thm}
	Let $(L,\preceq)$ be a complete lattice and $(L',\preceq')$ be a poset.
	A function $f\colon L\longrightarrow L'$ preserves arbitrary suprema
	if and only if there exists a function $u\colon L'\longrightarrow L$
	such that $f\dashv u$.
\end{proposition}

As the first application of Proposition~\ref{prop:adj_funct_thm}, 
observe that in any quantale $\cat{Q}=(Q,\preceq,$ $\one,\tms)$
there are two \emph{residuation} operations:
for any $x\in Q$, the function $(-)\tms x\colon Q\longrightarrow Q$
preserves arbitrary suprema, and hence has a right adjoint 
$(-)\mot x\colon Q\longrightarrow Q$
called the \defemph{right extension along $x$};
similarly, for any $y\in Q$ the function 
$y\tms (-)$ has a right adjoint $y\mto (-)$,
called the \defemph{right lifting along $y$}.
Of course, in a commutative quantale the right extensions and right liftings 
coincide.
The defining adjointness relations are:
\begin{equation}\label{eqn:residual_adjointness}
y\preceq z\mot x \iff y\tms x\preceq z \iff x\preceq y\mto z.
\end{equation}
Focusing on the leftmost and rightmost formulas of 
(\ref{eqn:residual_adjointness}), we obtain
\[
z\mot x\succeq y \iff x\preceq y\mto z.
\]
That is, $z\mot (-)\colon Q\longrightarrow Q$ (regarded as a 
function from the poset $\cat{Q}=(Q,\preceq)$ to its
dual $\cat{Q}^\op=(Q,\succeq)$) is the left 
adjoint of $(-)\mto z\colon 
Q\longrightarrow {Q}$ (from $\cat{Q}^\op$ to $\cat{Q}$)
for any $z\in Q$.
The three types of adjunctions
\begin{equation*}
\begin{tikzpicture}[baseline=-\the\dimexpr\fontdimen22\textfont2\relax ]
\node (L) at (0,0) {$\cat{Q}$};
\node (R) at (3,0) {$\cat{Q}$};

\draw [->, bend left=20] (L) to node [auto,labelsize] {$(-)\tms x$} (R);
\draw [<-, bend right=20] (L) to node [auto,labelsize,swap] {$(-)\mot x$} 
(R);

\node [rotate=-90]at (1.5,0) {$\dashv$};
\end{tikzpicture}
\qquad
\begin{tikzpicture}[baseline=-\the\dimexpr\fontdimen22\textfont2\relax ]
\node (L) at (0,0) {$\cat{Q}$};
\node (R) at (3,0) {$\cat{Q}$};

\draw [->, bend left=20] (L) to node [auto,labelsize] {$y\tms (-)$} (R);
\draw [<-, bend right=20] (L) to node [auto,labelsize,swap] {$y\mto (-)$} 
(R);

\node [rotate=-90]at (1.5,0) {$\dashv$};
\end{tikzpicture}
\qquad
\begin{tikzpicture}[baseline=-\the\dimexpr\fontdimen22\textfont2\relax ]
\node (L) at (0,0) {$\cat{Q}$};
\node (R) at (3,0) {$\cat{Q}^\op$};

\draw [->, bend left=20] (L) to node [auto,labelsize] {$z\mot (-)$} (R);
\draw [<-, bend right=20] (L) to node [auto,labelsize,swap] {$(-)\mto z$} 
(R);

\node [rotate=-90]at (1.5,0) {$\dashv$};
\end{tikzpicture}
\end{equation*}
are fundamental in the theory of quantales. 

\medskip

We conclude this section with some examples of quantales.

\begin{example}[\cite{Lawvere_metric}]
	The {two-element quantale} $\TV=(\{\bot,\top\}, 
	\vdash ,\top, \wedge)$.
	The underlying poset of this quantale consists of $\top$ for ``truth'' and $\bot$ for ``falsity'', ordered by the entailment relation $\vdash$, so that $\bot \vdash \top$.
	The monoid structure is given by conjunction $\wedge$.
\end{example}

\begin{example}[\cite{Lawvere_metric}]
	The \defemph{Lawvere quantale} $\nonnegRmin=([0,\infty], \geq, 0,+)$.
	Here, $([0,\infty],\geq)$ is the poset $([0,\infty),\geq)$ of all nonnegative real numbers ordered by the \emph{opposite} $\geq$ of the usual order $\leq$, extended with the \emph{least} element $\infty$.
	The $+$ operation is the extension of addition for nonnegative real numbers to $[0,\infty]$ so that $x+\infty=\infty+x=\infty$ for all $x\in[0,\infty]$. (This extension is forced by the axioms of quantale.)
	The right extension/right lifting is given by an extension of the
	\emph{truncated subtraction $\dotminus$}, defined by
	$u \dotminus t = \max\{u-t, 0\}$ for nonnegative real numbers $t$ and $u$.
	Precisely, for $y,z\in[0,\infty]$, 
	\[
	z\mot y=y\mto z=\begin{cases}
	z\dotminus y & \text{ if } y,z\in[0,\infty)\\
	0            & \text{ if } y=\infty\\
	\infty       & \text{ if } z=\infty \text{ and } y\in [0,\infty).
	\end{cases}
	\]
	This quantale is introduced in \cite{Lawvere_metric} for a categorical approach to the theory of metric spaces. 
\end{example}

\begin{example}[\cite{Lawvere_metric}]
	$\nonnegRminmax=([0, \infty],\geq, 0, \max)$.
	Its underlying poset $([0,\infty],\geq)$ is the same as that of $\nonnegRmin$. We take the binary max operation with respect to the usual ordering $\leq$, namely the binary meet operation with respect to $\geq$, as the multiplication. The right extension/right lifting is given by 
	\[
	z\mot y=y\mto z=\begin{cases}
	0 \ \text{ if }y\geq z\\
	z \ \text{ otherwise.}
	\end{cases}
	\]
	This quantale is related to (a generalisation of) ultrametric spaces.
	
	We remark that more generally, any \defemph{locale}, i.e., a complete lattice in which the binary meet operation $\wedge$ satisfies the infinitary distributive law $(\bigvee_{i\in I} y_i)\wedge x=\bigvee_{i\in I}(y_i\wedge x)$, acquires a quantale structure with $\wedge$ as the multiplication;
	indeed quantales were first introduced as a quantum theoretic generalisation of locales \cite{Mulvey_and}.
	The poset $([0,\infty],\geq)$, or more generally any totally ordered complete lattice, is a locale.
\end{example}

\begin{example}
	Let $\cat{M}=(M,e,\cdot)$ be a monoid. The
	\defemph{free quantale generated by $\cat{M}$}
	is $\Pow \cat{M}=(\Pow M,{\subseteq},\{e\}, \cdot)$,
	where $\Pow M$ is the power set of $M$ and the multiplication $\cdot$ on $\Pow 
	M$ is the unique supremum-preserving extension of the original multiplication on $M$, which is given by
	\[
	A\cdot B = \{\,a\cdot b\mid a\in A, b\in B\,\}
	\]
	for all $A,B\in\Pow M$.
	Unlike the previous examples, this quantale is not commutative unless $\cat{M}$ is.
\end{example}

\begin{example}\label{ex:quantale_binary_rel}
	Let $A$ be a set. The poset $(\Pow (A\times A),\subseteq)$
	of all binary relations on $A$ admits a quantale structure
	$(\Pow(A\times A),\subseteq, I_{A},\circ)$, where $I_A$ denotes the diagonal relation on $A$ and $\circ$ denotes composition of relations.
	This quantale is not commutative in general. 
\end{example}

\section{\texorpdfstring{$\cat{Q}$}{Q}-categories}
\label{sec:Q-cat}
In this section, we introduce $\cat{Q}$-categories for a quantale $\cat{Q}$. They are instances of the well-established notion of enriched category \cite{Kelly:enriched}.
\emph{Throughout the rest of this paper,  $\cat{Q}=(Q,\preceq_\cat{Q},\one_\cat{Q},\tms_\cat{Q})$ denotes an arbitrary quantale, unless otherwise specified.}

\begin{definition}
	A \defemph{$\cat{Q}$-category} $\cat{C}$ consists of:
	\begin{description}
		\item[(CD1)] a set $\ob{\cat{C}}$ of \defemph{objects};
		\item[(CD2)] for each $c,c'\in\ob{\cat{C}}$, an element $\cat{C}(c,c')\in {Q}$
	\end{description}
	satisfying the following axioms:
	\begin{description}
		\item[(CA1)] for each $c\in\ob{\cat{C}}$, $\one_\cat{Q} \preceq_\cat{Q} \cat{C}(c,c)$;
		\item[(CA2)] for each $c,c',c''\in\ob{\cat{C}}$, $\cat{C}(c',c'')\tms_\cat{Q} \cat{C}(c,c')\preceq_\cat{Q} 
		\cat{C}(c,c'')$.
	\end{description}
	We also write $c\in\cat{C}$ for $c\in\ob{\cat{C}}$.
\end{definition}

\begin{example}\label{ex:2-cat}
	In the case $\cat{Q}=\TV$,  we may identify the data of a $\TV$-category $\cat{C}=(\ob{\cat{C}}, (\cat{C}(c,c'))_{c,c'\in\ob{\cat{C}}})$  with a set $\ob{\cat{C}}$ equipped with a binary relation $\preceq_\cat{C}$ on it (defined as the set of all pairs $(c,c')\in\ob{\cat{C}}\times\ob{\cat{C}}$ with $\cat{C}(c,c')=\top$).
	Axioms (CA1) and (CA2) for a $\TV$-category then translate to reflexivity and transitivity of $\preceq_\cat{C}$ respectively, hence a $\TV$-category is nothing but a \emph{preordered set}.
\end{example}

\begin{example}\label{ex:nonnegRmin_cat_gms}
	In the case $\cat{Q}=\nonnegRmin$, we may regard $\nonnegRmin$-categories as {generalised metric spaces} \cite{Lawvere_metric}.
	Objects of an $\nonnegRmin$-category $\cat{C}$ are thought of as \emph{points} and the element $\cat{C}(c,c')\in[0,\infty]$ as the \emph{distance from $c$ to $c'$}.
	Notice that the axioms for $\nonnegRmin$-category indeed translate to some of the axioms for metric spaces:
	\begin{description}
		\item[(CA1)] for each $c\in\ob{\cat{C}}$, $0\geq \cat{C}(c,c)$ (that is, $\cat{C}(c,c)=0$); and
		\item[(CA2)] for each $c,c',c''\in\ob{\cat{C}}$, $\cat{C}(c',c'')+ \cat{C}(c,c')\geq \cat{C}(c,c'')$ (the triangle inequality).
	\end{description}
	We call $\nonnegRmin$-categories \defemph{Lawvere metric spaces}.
	Every metric space is a Lawvere metric space, but not conversely. Lawvere metric spaces are more general than metric spaces in the following three aspects:
	\begin{itemize}
		\item distance may take $\infty$;
		\item distance is non-symmetric (or \emph{directed}), i.e., $\cat{C}(c,c')$ may be different from $\cat{C}(c',c)$; and
		\item $\cat{C}(c,c')=\cat{C}(c',c)=0$ does not necessarily imply $c=c'$. 
	\end{itemize}
\end{example}

\begin{example}
	Similarly, $\nonnegRminmax$-categories may be regarded as {generalised ultrametric spaces}; note that axiom (CA2) now reads:
	\begin{description}
		\item[(CA2)] for each $c,c',c''\in\ob{\cat{C}}$, $\max\{\,\cat{C} (c',c''), \cat{C}(c,c') \,\}\geq \cat{C} (c,c'')$. 
	\end{description}
\end{example}

\begin{example}
	Let $\cat{M}=(M,e,\cdot)$ be a monoid. A $\Pow \cat{M}$-category $\cat{C}$ has, for each pair $c,c'\in\ob{\cat{C}}$, a subset $\cat{C}(c,c')\subseteq M$. These subsets must satisfy:
	\begin{description}
		\item[(CA2)] for each $c\in\ob{\cat{C}}$, $e\in \cat{C}(c,c)$; and 
		\item[(CA2)] for each $c,c',c''\in\ob{\cat{C}}$, $n\in \cat{C} (c',c'')$ and $m\in \cat{C}(c,c')$, $n\cdot m\in\cat{C} (c,c'')$. 
	\end{description}
	It follows that a $\Pow \cat{M}$-category can be identified with an ordinary category $\cat{C}$ equipped with a \emph{faithful} functor $\cat{C}\longrightarrow\cat{M}$, where the monoid $\cat{M}$ is regarded as a one-object category.
	
	In fact, this example can be vastly generalised. For any (ordinary) \emph{category} $\cat{B}$, we can construct the free \emph{quantaloid} $\Pow\cat{B}$ over it; quantaloids \cite{Rosenthal_quantaloid_automata} are a many-object version of quantales, just like categories can be seen as a many-object version of monoids.
	It turns out that a $\Pow\cat{B}$-category corresponds to a category $\cat{C}$ equipped with a faithful functor $\cat{C}\longrightarrow\cat{B}$ \cite{Garner_topological}. 
	Theorem \ref{thm:main} is known to generalise to quantaloid-enriched categories \cite{Stubbe_note,Stubbe_double,Shen_Tholen_topological}, and (skeletal and) complete/injective $\Pow\cat{B}$-categories correspond to \emph{topological functors over $\cat{B}$}; see \cite{Garner_topological} for a characterisation of topological functors in terms of completeness, and see \cite{Brummer_Hoffmann_external,Herrlich_initial}  for that in terms of injectivity.
	In order to keep the paper accessible to a wider audience, in this paper we shall not pursue quantaloid-enriched categories any further. On the other hand, although our main examples of base quantales are commutative, we shall not assume commutativity so that our arguments can be easily generalised to the case of quantaloids. 
\end{example}

\medskip

Let $\cat{C}$ be a $\cat{Q}$-category. We define a preorder relation $\preceq_\cat{C}$ on $\cat{C}$ as 
\[
c\preceq_\cat{C}c' \iff I_\cat{Q}\preceq_\cat{Q}\cat{C}(c,c').
\]
(Note that the notation $\preceq_\cat{C}$ agrees with the one introduced in Example \ref{ex:2-cat}.)
Two objects $c,c'\in\cat{C}$ are said to be \defemph{isomorphic} if both $c\preceq_\cat{C} c'$ and $c'\preceq_\cat{C} c$ hold.
Isomorphic objects behave exactly in the same manner: if $c,c'\in\cat{C}$ are isomorphic then for every object $d\in\cat{C}$, we have $\cat{C}(c,d)=\cat{C}(c',d)$ and $\cat{C}(d,c)=\cat{C}(d,c')$.
We call $\cat{C}$ \defemph{skeletal} if isomorphic objects in $\cat{C}$ are equal, i.e., if the induced preorder relation $\preceq_\cat{C}$ on $\ob{\cat{C}}$ is actually a partial order relation.

A skeletal $\TV$-category is a poset, and a skeletal $\nonnegRmin$- or $\nonnegRminmax$-category $\cat{C}$ satisfies the condition that for all $c,c'\in\cat{C}$, $\cat{C}(c,c')=\cat{C}(c',c)=0$ implies $c=c'$.

\medskip

Let us move on to define completeness of $\cat{Q}$-categories.
Given a $\cat{Q}$-category $\cat{C}$, an object $c\in\cat{C}$ and an element $x\in Q$, an object $c'\in\cat{C}$ is said to be a \defemph{power of $c$ by $x$} if for any $d\in\cat{C}$, the equation
\[
\cat{C}(d,c')=x\mto \cat{C}(d,c)
\]
holds \cite[Section~3.7]{Kelly:enriched}.
Powers of $c$ by $x$ may or may not exist in $\cat{C}$, but when they exist they are unique up to isomorphism: if $c'$ is a power of $c$ by $x$, then an object $c''\in\cat{C}$ is also a power of $c$ by $x$ if and only if $c'$ and $c''$ are isomorphic. 
In particular, in a skeletal $\cat{Q}$-category powers are unique. We denote the power of $c$ by $x$ by $x \pitchfork c$.

There is also a dual notion of \defemph{copower of $c\in \cat{C}$ by $x\in Q$}, which is defined as an object $c'\in\cat{C}$ such that for any $d\in\cat{C}$, the equation
\[
\cat{C}(c',d)=\cat{C}(c,d)\mot x
\]
holds. The copower of $c$ by $x$ is denoted by $x\ast c$.

\begin{definition}[{\cite[Section 2]{Stubbe_tensor_cotensor}}]\label{def:completeness}
	A $\cat{Q}$-category $\cat{C}$ is said to be:
	\begin{itemize}
		\item \defemph{powered} if for any $c\in \cat{C}$ and $x\in Q$, the power $x\pitchfork c$ exists in $\cat{C}$; 
		\item \defemph{copowered} if for any $c\in\cat{C}$ and $x\in Q$, the copower $x\ast c$  exists in $\cat{C}$; 
		\item \defemph{order-complete} if the preordered set $(\ob{\cat{C}},\preceq_\cat{C})$ is complete (i.e., if its poset reflection\footnote{The \emph{poset reflection} of a preordered set $(P,\preceq)$ is the quotient of it by the equivalence relation $\preceq\cap\succeq$.} is a complete lattice); and
		\item \defemph{complete} if it is powered, copowered and order-complete. 
	\end{itemize}
\end{definition}

\medskip

Next we define morphisms between $\cat{Q}$-categories, called \emph{$\cat{Q}$-functors}.

\begin{definition}Let $\cat{C}$ and $\cat{D}$ be $\cat{Q}$-categories.
	\begin{enumerate}
		\item A \defemph{$\cat{Q}$-functor} $f\colon \cat{C}\longrightarrow\cat{D}$ is a function $f\colon \ob{\cat{C}}\longrightarrow\ob{\cat{D}}$
		such that for each $c,c'\in\cat{C}$, 
		\begin{equation}\label{eqn:Q-functor_ineq}
		\cat{C}(c,c')\preceq_\cat{Q} \cat{D}(f(c),f(c'))
		\end{equation}
		holds.
		\item A $\cat{Q}$-functor $f\colon\cat{C}\longrightarrow\cat{D}$ is \defemph{fully faithful} if for each $c,c'\in\cat{C}$, (\ref{eqn:Q-functor_ineq}) is satisfied with equality.
		We call fully faithful $\cat{Q}$-functors \defemph{embeddings} for short. 
	\end{enumerate}
\end{definition}
For any $\cat{Q}$-category $\cat{C}$ we have the \defemph{identity $\cat{Q}$-functor} $\id{\cat{C}}\colon\cat{C}\longrightarrow\cat{C}$ (given by the identity function on $\ob{\cat{C}}$), and 
$\cat{Q}$-functors are closed under composition. So $\cat{Q}$-categories and $\cat{Q}$-functors form an (ordinary) category $\enCat{\cat{Q}}$.
Note that an embedding $f\colon\cat{C}\longrightarrow\cat{D}$ of $\cat{Q}$-categories need not be injective as a function $f\colon\ob{\cat{C}}\longrightarrow\ob{\cat{D}}$, though embeddings out of a skeletal $\cat{C}$ {are} injective.

For example, a $\TV$-functor $f\colon \cat{C}\longrightarrow \cat{D}$ is a monotone map, and an $\nonnegRmin$- or $\nonnegRminmax$-functor $f\colon \cat{C}\longrightarrow \cat{D}$ is a nonexpansive map.
Embeddings specialise to order embeddings and isometric embeddings respectively.

We say that a $\cat{Q}$-category $\cat{E}$ is \defemph{injective} (with respect to embeddings) if, whenever we have $\cat{Q}$-categories $\cat{C}$ and $\cat{D}$, a $\cat{Q}$-functor $f\colon \cat{C}\longrightarrow\cat{E}$, and an embedding $i\colon \cat{C}\longrightarrow \cat{D}$, there exists a (not necessarily unique) $\cat{Q}$-functor $g\colon \cat{D}\longrightarrow\cat{E}$ such that $f=g\circ i$.

Thus we have defined all terms appearing in Theorem \ref{thm:main}. In fact, we can already prove the easier direction.

\begin{lemma}\label{lem:complete_injective}
	A skeletal and complete $\cat{Q}$-category is injective.
\end{lemma}
\begin{proof}
	Let $\cat{E}$ be a skeletal and complete $\cat{Q}$-category. Given a diagram as in 
	\begin{equation*}
	\begin{tikzpicture}[baseline=-\the\dimexpr\fontdimen22\textfont2\relax ]
	\node (TL) at (0,1) {$\cat{C}$};
	\node (BL) at (0,-1) {$\cat{D}$};
	\node (TR) at (2,1) {$\cat{E}$};
	
	\draw [->] (TL) to node [auto,labelsize] {$f$} (TR);
	\draw [->] (TL) to node [auto,labelsize,swap] {$i$} (BL);
	\end{tikzpicture}
	\end{equation*}
	we may define a $\cat{Q}$-functor $g\colon \cat{D}\longrightarrow\cat{E}$ as the (pointwise) left (or right) Kan extension of $f$ along $i$, namely
	\begin{align*}
	g(d) &= \Lan_i f(d)=\bigvee_{c\in\cat{C}}\cat{D}(i(c), d)\ast f(c)\\
	(\text{or }g(d)&=\Ran_i f(d)=\bigwedge_{c\in\cat{C}} \cat{D}(d,i(c))\pitchfork f(c)).
	\end{align*}
	Then, provided that $i$ is an embedding, $f(c)$ and $(g\circ i)(c)$ are isomorphic for all $c\in\cat{C}$ \cite[Proposition 6.7]{Stubbe_quantaloid_dist}; but since $\cat{E}$ is skeletal, isomorphic objects are necessarily equal, so $f=g\circ i$.
	(Incidentally, $\Lan_i f$ and $\Ran_i f$ are respectively the least and greatest $g$ such that $f=g\circ i$; that is, a $\cat{Q}$-functor $g\colon \cat{D}\longrightarrow\cat{E}$ satisfies $f=g\circ i$ if and only if $\Lan_i f\preceq g\preceq \Ran_i f$ with respect to the pointwise order $\preceq$ induced from $\preceq_\cat{E}$.)
\end{proof}

\section{The MacNeille completion of a \texorpdfstring{$\cat{Q}$}{Q}-category}
\label{sec:Isbell}
In this section we explain the MacNeille completion of a $\cat{Q}$-category. 
We start with some preparation.
\begin{definition}[{Cf.~\cite[Proposition 6.1]{Stubbe_quantaloid_dist}}]
	Let $\cat{C}$ be a $\cat{Q}$-category. 
	The $\cat{Q}$-category $\psh\cat{C}$ of presheaves over $\cat{C}$ is defined as follows.
	\begin{itemize}
		\item An object is a \defemph{presheaf over $\cat{C}$}, that is a family $P=(Pc)_{c\in\cat{C}}$ of elements of ${Q}$ satisfying the inequality $Pc'\tms \cat{C}(c,c')\preceq_\cat{Q}Pc$ for each $c,c'\in\cat{C}$.
		\item The element $\psh\cat{C}(P,P')$ of ${Q}$ is given by $\bigwedge_{c\in\cat{C}}P'c\mot Pc$.
	\end{itemize}
	
	Dually, the $\cat{Q}$-category $\copsh\cat{C}$ of copresheaves over $\cat{C}$ is defined as follows.
	\begin{itemize}
		\item An object is a \defemph{copresheaf over $\cat{C}$}, that is a family $R=(Rc)_{c\in\cat{C}}$ of elements of ${Q}$ satisfying the inequality $\cat{C}(c,c')\tms Rc\preceq_\cat{Q} Rc'$ for each $c,c'\in\cat{C}$.
		\item The element $\copsh\cat{C}(R,R')$ of ${Q}$ is given by $\bigwedge_{c\in\cat{C}}R'c\mto Rc$.
	\end{itemize}
\end{definition}

For any $\cat{Q}$-category $\cat{C}$, there are well-known embeddings $y_\cat{C}\colon\cat{C}\longrightarrow\psh\cat{C}$ and $y^\dagger_\cat{C}\colon\cat{C}\longrightarrow\copsh\cat{C}$ called the \defemph{Yoneda} and \defemph{co-Yoneda embeddings} respectively; they are defined as $y_\cat{C}(c)=(\cat{C}(c',c))_{c'\in\cat{C}}$ and $y^\dagger_\cat{C}(c)=(\cat{C}(c,c'))_{c'\in\cat{C}}$.

For example, when $\cat{Q}=\TV$, $\psh\cat{C}$ can be understood as the poset of all \emph{lower sets} of $\cat{C}$ (i.e., subsets $P\subseteq \ob{\cat{C}}$ such that $c'\in P$ and $c\preceq_\cat{C} c'$ imply $c\in P$), ordered by inclusion. The Yoneda embedding maps an element $c\in\cat{C}$ to the principal lower set $\downarrow\! c=\{\,c'\in\cat{C}\mid c'\preceq_\cat{C} c\,\}$ generated by it. Dually, $\copsh\cat{C}$ is the poset of all \emph{upper sets} of $\cat{C}$ ordered by the \emph{opposite} of inclusion, and the co-Yoneda embedding maps $c\in\cat{C}$ to the principal upper set $\uparrow\! c$.

There exists a pair of canonical $\cat{Q}$-functors 
\begin{equation}\label{eqn:Isbell_adj}
\cat{C}\mot (-)\colon \psh\cat{C}\longrightarrow\copsh\cat{C}\quad  \text{ and }\quad (-)\mto\cat{C}\colon \copsh\cat{C}\longrightarrow\psh\cat{C}.
\end{equation}
The functor $\cat{C}\mot (-)\colon\psh\cat{C}\longrightarrow\copsh\cat{C}$ maps a presheaf $P=(Pc)_{c\in\cat{C}}$ to the copresheaf $\cat{C}\mot P=((\cat{C}\mot P)c)_{c\in\cat{C}}$ defined as 
\[
(\cat{C}\mot P)c=\bigwedge_{c'\in\cat{C}}\cat{C}(c',c)\mot Pc';
\]
the functor $(-)\mto \cat{C}$ maps $R\in\copsh\cat{C}$ to $R\mto \cat{C}\in\psh\cat{C}$ defined as 
\[
(R\mto \cat{C})c=\bigwedge_{c'\in\cat{C}}Rc'\mto \cat{C}(c,c').
\]
The functors \eqref{eqn:Isbell_adj} form a \emph{$\cat{Q}$-adjunction} $\cat{C}\mot (-)\dashv (-)\mto \cat{C}$, in the sense that $\psh\cat{C}(P,R\mto \cat{C})=\copsh\cat{C}(\cat{C}\mot P, R)$ for all $P\in\psh\cat{C}$ and $R\in\copsh\cat{C}$. This can be checked as follows: 
\begin{align*}
\psh\cat{C}(P,R\mto \cat{C})&=\bigwedge_{c,c'\in\cat{C}}\Big(Rc'\mto \cat{C}(c,c')\Big)\mot Pc\\
&=\bigwedge_{c,c'\in\cat{C}}Rc'\mto \Big(\cat{C}(c,c')\mot Pc\Big)\\
&=\copsh\cat{C}(\cat{C}\mot P,R).\\
\end{align*}
This adjunction is called the \defemph{Isbell adjunction} \cite{Shen_Zhang_Isbell_Kan,Garner_topological}.

The \defemph{MacNeille completion}  \cite{Garner_topological,Shen_thesis} $\Isb\cat{C}$ of $\cat{C}$ is the $\cat{Q}$-category defined as follows.
	\begin{description}
	\item[(CD1)] An object is a pair $(P,R)$ of a presheaf $P\in\psh\cat{C}$ and a copresheaf $R\in\copsh\cat{C}$ such that $P=R\mto \cat{C}$ and $R=\cat{C}\mot P$ hold.
	\item[(CD2)] Given two objects $(P,R)$ and $(P',R')$, the element \[\Isb\cat{C}((P,R),(P',R'))\in Q\] is defined as 
	$
	\psh\cat{C}(P,P')
	$,
	or equivalently as 
	$
	\copsh\cat{C}(R,R');
	$
	indeed, we have 
	\begin{equation*}
	\psh\cat{C}(P,P')=\psh\cat{C}(P,R'\mto \cat{C})
	=\copsh\cat{C}(\cat{C}\mot P,R')
	=\copsh\cat{C}(R,R').
	\end{equation*}
\end{description}
	Hence we have natural embeddings $p_\cat{C}\colon\Isb\cat{C}\longrightarrow\psh\cat{C}$ and $p^\dagger_\cat{C}\colon\Isb\cat{C}\longrightarrow\copsh\cat{C}$ defined by projections.
	The Yoneda (resp.~co-Yoneda) embedding factors through $p_\cat{C}$ (resp.~$p^\dagger_\cat{C}$), so we have a canonical embedding
	$i_\cat{C}\colon\cat{C}\longrightarrow\Isb\cat{C}$ which maps each $c\in\cat{C}$ to $(\cat{C}(-,c),\cat{C}(c,-))\in\Isb\cat{C}$.
	We summarise the situation in the diagram below.
	\[
	\begin{tikzpicture}[baseline=-\the\dimexpr\fontdimen22\textfont2\relax ]
	\node (L) at (3,1.5) {$\psh\cat{C}$};
	\node (R) at (3,-1.5) {$\copsh\cat{C}$};
	\node (M) at (0,0) {$\Isb\cat{C}$};
	\node (B) at (-3,0) {$\cat{C}$};
	
	\draw [->, transform canvas={xshift=-6}] (L) to node [auto,labelsize,swap] {$\cat{C}\mot (-)$} (R);
	\draw [<-, transform canvas={xshift=6}] (L) to node [auto,labelsize] {$(-)\mto \cat{C}$} (R);
	\draw [>->] (M) to node [auto,labelsize] {$p_\cat{C}$} (L);
	\draw [>->] (M) to node [auto,swap,labelsize] {$p^\dagger_\cat{C}$} (R);
	\draw [>->,bend left=15] (B) to node [auto,labelsize] {$y_\cat{C}$} (L);
	\draw [>->,bend right=15] (B) to node [auto,swap,labelsize] {$y^\dagger_\cat{C}$} (R);
	\draw [>->] (B) to node [auto,labelsize] {$i_\cat{C}$} (M);
	
	\path(L) to node {$\dashv$} (R);
	\end{tikzpicture}
	\]	

\begin{proposition}[\cite{Shen_thesis,Garner_topological}]\label{prop:Isb_skeletal_complete}
	Let $\cat{C}$ be a $\cat{Q}$-category. The MacNeille completion $\Isb\cat{C}$ is skeletal and complete.
\end{proposition}
\begin{proof}
	This is an immediate consequence of the fact that $\psh\cat{C}$ is skeletal and complete, and that $p_\cat{C}$ has a left adjoint, thus making $\MN\cat{C}$ a full reflective subcategory of $\psh\cat{C}$. See e.g., \cite[Proposition 7.6 (a)]{Garner_topological}.
\end{proof}

\section{A formal theory of injective envelopes}
\label{sec:formal_theory}

In this section, we recall the notion of injective envelope and its basic properties \cite{Adamek_et_al_inj}.
Throughout this section, let $\cat{X}$ be an (ordinary) category and $\cat{H}$ be a class of morphisms in $\cat{X}$, whose elements are called \defemph{embeddings}. We make no assumptions on $\cat{X}$ and $\cat{H}$, unless otherwise specified.
An example to bear in mind is the case where $\cat{X}=\enCat{\cat{Q}}$ and $\cat{H}$ is the class of all fully faithful $\cat{Q}$-functors.

\begin{definition}[{\cite[Definitions 2.1]{Adamek_et_al_inj}}]\label{def:inj_env}
	\begin{enumerate}
		\item An object $E$ of $\cat{X}$ is \defemph{injective} if, whenever we have objects $C$ and $D$, a morphism $f\colon C\longrightarrow D$, and an embedding $i\colon C\longrightarrow D$, there exists a (not necessarily unique) morphism $g\colon D\longrightarrow E$ such that $f=g\circ i$.
		\item A morphism $f\colon C\longrightarrow D$ in $\cat{X}$
		is called an \defemph{essential embedding} if: (i)~$f$ is an embedding, and (ii)~for any object $E$ and morphism $g\colon D\longrightarrow E$, if $g\circ f$ is an embedding then so is $g$.
		\item An \defemph{injective envelope} of an object ${C}$ of $\cat{X}$ is a pair $(D,f)$ consisting of an injective object $D$ and an essential embedding $f\colon {C}\longrightarrow {D}$. 
	\end{enumerate}
\end{definition}

Injective envelopes of an object are unique up to isomorphisms.\footnote{However, note that the nature of this ``uniqueness'' is quite different from that for usual categorical notions determined by their universal properties. We also remark that the operation of taking the injective envelopes of objects does not easily extend to morphisms \cite{Adamek_et_al_inj}; cf.~\cite{Shen_Zhang_Isbell_Kan} and \cite[Section 7]{Garner_topological}.}
\begin{lemma}[{\cite[Remarks 2.2 (2)]{Adamek_et_al_inj}}]\label{lem:inj_hull_uniqueness}
	Let $C$ be an object of $\cat{X}$, and $({D},f)$ and $({D'},f')$ be injective envelopes of ${C}$. Then there exists an isomorphism $g\colon {D}\longrightarrow {D'}$ such that $g\circ f=f'$.
\end{lemma}
\begin{proof}
	By the injectivity of ${D'}$, we obtain a morphism $g$ as in the following commutative diagram.
	\begin{equation*}
	\begin{tikzpicture}[baseline=-\the\dimexpr\fontdimen22\textfont2\relax ]
	\node (TL) at (0,1) {${C}$};
	\node (BL) at (0,-1) {${D}$};
	\node (TR) at (2,1) {${D'}$};
	
	\draw [>->] (TL) to node [auto,labelsize] {$f'$} (TR);
	\draw [>->] (TL) to node [auto,labelsize,swap] {$f$} (BL);
	\draw [->,dashed] (BL) to node [auto,labelsize,swap] {$g$} (TR);
	\end{tikzpicture}
	\end{equation*}
	We claim that {\emph{any} morphism $g$ between injective envelopes as above (i.e., commuting with the essential embeddings) is an isomorphism.}
	Since $f$ is an essential embedding and $f'=g\circ f$ is an embedding, it follows that $g$ is also an embedding.
	Using the injectivity of ${D}$, we obtain a morphism $h$ as below.
	\[
	\begin{tikzpicture}[baseline=-\the\dimexpr\fontdimen22\textfont2\relax ]
	\node (TL) at (0,1) {${D}$};
	\node (BL) at (0,-1) {${D'}$};
	\node (TR) at (2,1) {${D}$};
	
	\draw [->] (TL) to node [auto,labelsize] {$\id{{D}}$} (TR);
	\draw [>->] (TL) to node [auto,labelsize,swap] {$g$} (BL);
	\draw [->,dashed] (BL) to node [auto,labelsize,swap] {$h$} (TR);
	\end{tikzpicture}
	\]
	So $g$ is a section (split monomorphism) whereas $h$ is a retraction (split epimorphism).
	Precomposing $f$ with the above diagram, we obtain the following.
	\[
	\begin{tikzpicture}[baseline=-\the\dimexpr\fontdimen22\textfont2\relax ]
	\node (TL) at (0,1) {${C}$};
	\node (BL) at (0,-1) {${D'}$};
	\node (TR) at (2,1) {${D}$};
	
	\draw [>->] (TL) to node [auto,labelsize] {$f$} (TR);
	\draw [>->] (TL) to node [auto,labelsize,swap] {$f'$} (BL);
	\draw [->,dashed] (BL) to node [auto,labelsize,swap] {$h$} (TR);
	\end{tikzpicture}
	\]
	So $h$ is also a morphism between injective envelopes. Iterating the same argument as above, we see that $h$ is a section; hence $h$ is an isomorphism and so is its section, $g$.
\end{proof}

\begin{corollary}\label{cor:inj_env_iso}
	Suppose that the class $\cat{H}$ of embeddings contains all identity morphisms of $\cat{X}$.
	Let $C$ be an object of $\cat{X}$ and $(D,f)$ be an injective envelope of $C$. Then $C$ is injective if and only if $f$ is an isomorphism.
\end{corollary}
\begin{proof}
	If $C$ is injective, then by the assumption, $(C,\id{C})$ is an injective envelope of $C$. So by Lemma~\ref{lem:inj_hull_uniqueness} there exists an isomorphism $g\colon C\longrightarrow D$ such that $g\circ \id{C}=f$. Hence $g=f$ and $f$ is an isomorphism.
	
	Conversely, the class of all injective objects is clearly closed under isomorphism.
\end{proof}

The injective envelope is defined by the complementary properties of \emph{essentialness} of the embedding and \emph{injectivity} of the codomain. In fact it is ``extremal'' with respect to these two properties, in the following sense.

\begin{proposition}[{Cf.~\cite[Proposition 2]{Banaschewski_Bruns}}]\label{prop:inj_env_property}
	Let ${C}$ be an object of $\cat{X}$ and $(E,g)$ be its injective envelope.
	\begin{enumerate}
		\item For any \emph{essential} embedding $f\colon {C}\longrightarrow {D}$, there exists a (not necessarily unique) embedding $i\colon {D}\longrightarrow E$ with $i\circ f=g$.
		\item For any embedding $h\colon C\longrightarrow{F}$ into an \emph{injective} ${F}$, there exists a (not necessarily unique) embedding $k\colon E\longrightarrow{F}$ with $k\circ g=h$.
	\end{enumerate}
	\[
	\begin{tikzpicture}[baseline=-\the\dimexpr\fontdimen22\textfont2\relax ]
	\node (TL) at (0,1) {${C}$};
	\node (BLL) at (-2,-1) {$D$};
	\node (BL) at (0,-1) {$E$};
	\node (TR) at (2,-1) {$F$};
	
	\draw [>->] (TL) to node [auto, labelsize,swap] {$f$} (BLL);
	\draw [>->] (TL) to node [auto,labelsize] {$h$} (TR);
	\draw [>->] (TL) to node [auto,labelsize] {$g$} (BL);
	\draw [>->,dashed] (BL) to node [auto,labelsize,swap] {$k$} (TR);
	\draw[>->,dashed] (BLL) to node [auto,labelsize,swap] {$i$} (BL);
	\end{tikzpicture}
	\]
\end{proposition}

\section{The MacNeille completion is the injective envelope}\label{sec:complete_iff_injective}
In this section we prove Theorem \ref{thm:main_envelope} and, using that, Theorem \ref{thm:main}.
Whenever we use the notions introduced in the previous section, we take $\cat{X}=\enCat{\cat{Q}}$ and $\cat{H}$ to be  the class of all fully faithful $\cat{Q}$-functors.

The key step is to give an intrinsic characterisation of essential embeddings.
For each $\cat{Q}$-functor $f\colon \cat{C}\longrightarrow\cat{D}$, we have $\cat{Q}$-functors $f^\ast\colon \cat{D}\longrightarrow\psh\cat{C}$ and $f_\ast \colon \cat{D}\longrightarrow\copsh\cat{C}$,
defined as 
$f^\ast (d)= (\cat{D}(f(c), d))_{c\in\cat{C}}$
and $f_\ast (d) =(\cat{D}(d,f(c)))_{c\in\cat{C}}$
respectively. 
We call $f$ \defemph{dense} if $f^\ast$ is an embedding, and \defemph{codense} if $f_\ast$ is an embedding \cite[Chapter 5]{Kelly:enriched}.

For any $\cat{Q}$-category $\cat{C}$, the Yoneda embedding $y_\cat{C}\colon\cat{C}\longrightarrow\psh\cat{C}$ is $(\id{\cat{C}})^\ast$, whereas the co-Yoneda embedding $y^\dagger_\cat{C}\colon\cat{C}\longrightarrow\copsh\cat{C}$ is $(\id{\cat{C}})_\ast$. In particular, $\id{\cat{C}}$ is both dense and codense. Moreover, $(y_\cat{C})^\ast \colon \psh\cat{C}\longrightarrow\psh\cat{C}$ is the identity $\cat{Q}$-functor $\id{\psh\cat{C}}$ (the Yoneda lemma), so $y_\cat{C}$ is dense (but in general not codense). 
Dually, $y^\dagger_\cat{C}$ is codense (but in general not dense).
(Incidentally, $(y_\cat{C})_\ast=\cat{C}\mot(-)\colon \psh\cat{C}\longrightarrow\copsh\cat{C}$ and $(y^\dagger_\cat{C})^\ast=(-)\mto\cat{C}\colon\copsh\cat{C}\longrightarrow\psh\cat{C}$ \cite[Remark 6.7]{Garner_topological}.)

\emph{The canonical embedding $i_\cat{C}\colon\cat{C}\longrightarrow\Isb\cat{C}$ of $\cat{C}$ into its MacNeille completion is both dense and codense} (\cite[Theorem 4.16]{Shen_Zhang_Isbell_Kan}, \cite[Proposition 7.6]{Garner_topological} and \cite[Theorem 6.5]{Lai_Shen_fixed}), because $(i_\cat{C})^\ast\colon \Isb\cat{C}\longrightarrow\psh\cat{C}$
and $(i_\cat{C})_\ast\colon \Isb\cat{C}\longrightarrow\copsh\cat{C}$
coincide with the (fully faithful) projections $p_\cat{C}$ and $p^\dagger_\cat{C}$ respectively. 

\begin{proposition}[{Cf.~\cite[Lemma 3]{Banaschewski_Bruns}}]\label{prop:ess_emb_characterisation}
	A $\cat{Q}$-functor is an essential embedding if and only if it is a dense and codense embedding.
\end{proposition}
\begin{proof}
	Suppose that $f\colon \cat{C}\longrightarrow\cat{D}$ is an essential embedding. Then the composite $f^\ast \circ f\colon \cat{C}\longrightarrow\psh\cat{C}$ maps each $c\in \cat{C}$ to $(\cat{D}(fc',fc))_{c'\in\cat{C}}=(\cat{C}(c',c))_{c'\in\cat{C}}$, i.e., it is the Yoneda embedding $y_\cat{C}$.
	In particular, $f^\ast \circ f$ is an embedding and hence so is $f^\ast$, showing that $f$ is dense. A similar argument shows that $f$ is codense.
	
	Conversely, suppose that $f\colon\cat{C}\longrightarrow\cat{D}$ is a dense and codense embedding. Take any $\cat{Q}$-functor $g\colon \cat{D}\longrightarrow\cat{E}$ such that $g\circ f$ is an embedding. We aim to show that $g$ is also an embedding, namely that for each $d,d'\in\cat{D}$ we have $\cat{D}(d,d')=\cat{E}(gd,gd')$. Since $g$ is a $\cat{Q}$-functor, it suffices to show the inequality 
	\begin{equation}\label{eqn:EGDGD}
	\cat{E}(gd,gd')\preceq_\cat{Q}\cat{D}(d,d').
	\end{equation}
	
	Since $f$ is dense, we have 
	\begin{equation}\label{eqn:DDD}
	\cat{D}(d,d')=\psh\cat{C}(f^\ast d,f^\ast d')=\bigwedge_{c\in\cat{C}}\cat{D}(fc,d')\mot \cat{D}(fc,d).
	\end{equation}
	Since $f$ is codense, we have
	\begin{equation}\label{eqn:DFCD}
	\cat{D}(fc,d')=\copsh\cat{C}(f_\ast fc,f_\ast d')=\bigwedge_{c'\in\cat{C}}\cat{D}(d',fc')\mto \cat{D}(fc,fc').
	\end{equation}
	Substituting (\ref{eqn:DFCD}) into (\ref{eqn:DDD}), we obtain
	\begin{align*}
	\cat{D}(d,d')&=\bigwedge_{c\in\cat{C}}\Big(\bigwedge_{c'\in \cat{C}} \cat{D}(d',fc')\mto \cat{D}(fc,fc')\Big)\mot \cat{D}(fc,d)\\
	&= \bigwedge_{c,c'\in\cat{C}}\Big(\cat{D}(d',fc')\mto \cat{D}(fc,fc')\Big)\mot \cat{D}(fc,d)
	\end{align*}
	(cf.~\cite[Theorem 1]{Dress_tight_ext}).
	Hence to show (\ref{eqn:EGDGD}) it suffices to show, for each $c,c'\in\cat{C}$, 
	\[
	\cat{E}(gd,gd') \preceq_\cat{Q}\Big(\cat{D}(d',fc')\mto \cat{D}(fc,fc')\Big)\mot \cat{D}(fc,d).
	\]
	Using the adjointness relation (\ref{eqn:residual_adjointness}) twice, this is equivalent to
	\[
	\cat{D}(d',fc')\tms \cat{E}(gd,gd')\tms \cat{D}(fc,d)\preceq_\cat{Q} \cat{D}(fc,fc'),
	\]
	which can be checked easily as follows:
	\begin{align*}
	\cat{D}(d',fc')\tms \cat{E}(gd,gd')\tms \cat{D}(fc,d)&\preceq_\cat{Q} \cat{E}(gd',gfc')\tms \cat{E}(gd,gd')\tms \cat{E}(gfc,gd)\\
	&\preceq_\cat{Q}\cat{E}(gfc,gfc')\\
	&=\ \ \cat{C}(c,c')\\
	&=\ \ \cat{D}(fc,fc'). \qedhere
	\end{align*}
\end{proof}

Now we can show Theorem \ref{thm:main_envelope} claiming that for any $\cat{Q}$-category $\cat{C}$, $(\Isb\cat{C},i_\cat{C})$ is its injective envelope.
The $\cat{Q}$-category $\Isb\cat{C}$ is skeletal and complete by Proposition \ref{prop:Isb_skeletal_complete}, hence injective by Lemma \ref{lem:complete_injective}.
Since the $\cat{Q}$-functor $i_\cat{C}$ is a dense and codense embedding, it is an essential embedding by Proposition~\ref{prop:ess_emb_characterisation}.

Theorem \ref{thm:main} follows at once.
Since all identity $\cat{Q}$-functors are essential embeddings, by Corollary \ref{cor:inj_env_iso}, a $\cat{Q}$-category $\cat{C}$ is injective if and only if the embedding $i_\cat{C}\colon\cat{C}\longrightarrow\MN\cat{C}$ is an isomorphism. In particular, if $\cat{C}$ is injective, then it is isomorphic to the skeletal and complete $\Isb\cat{C}$, so $\cat{C}$ is also skeletal and complete. The converse has already been shown in Lemma \ref{lem:complete_injective}.

We remark that, being the injective envelope, the MacNeille completion enjoys the extremal properties described in Proposition~\ref{prop:inj_env_property};
cf.~\cite[Proposition 5.5.5]{Shen_thesis} and \cite[Proposition 7.6 (e)]{Garner_topological}.

\section{Isbell convexity as categorical completeness}\label{sec:Isb_conv}
Finally, in this section we clarify the relationship of Theorem \ref{thm:KKO} due to Kemajou, K{\"u}nzi and Otafudu~\cite[Theorem 1]{Kemajou_Kunzi_Otafudu_Isbell_hull}, and the $\cat{Q}=\nonnegRmin$ case of our theorem.

We first work over a general quantale $\cat{Q}$ and provide an alternative description of objects of the MacNeille completion. 
Given (possibly infinite) sets $A$ and $B$, a \defemph{$\cat{Q}$-matrix} from $A$ to $B$ is simply a function $X\colon A\times B\longrightarrow {Q}$ \cite{Betti_et_al_var_enr}.
We denote such a $\cat{Q}$-matrix as $X\colon A\pto B$. 
Given $\cat{Q}$-matrices $X\colon A\pto B$, $Y\colon B\pto C$ and $Z\colon A\pto C$, define the $\cat{Q}$-matrices:
\begin{itemize}
	\item $Y\tms X\colon A\pto C$ as $(Y\tms X)(a,c)=\bigvee_{b\in B}Y(b,c)\tms X(a,b)$;
	\item $Z\mot X\colon B\pto C$ as $(Z\mot X)(b,c)=\bigwedge_{a\in A}Z(a,c)\mot X(a,b)$; and 
	\item $Y\mto Z\colon A\pto B$ as $(Y\mto Z)(a,b)=\bigwedge_{c\in C}Y(b,c)\mto Z(a,c)$.
\end{itemize}
It is routine to check that these operations on $\cat{Q}$-matrices satisfy the adjointness relations analogous to \eqref{eqn:residual_adjointness}, namely 
\[
Y\preceq Z\mot X \iff Y\tms X\preceq Z \iff X\preceq Y\mto Z,
\] 
where we order $\cat{Q}$-matrices of a fixed domain and codomain by the pointwise order induced from $\preceq_\cat{Q}$.
Also, for any set $A$ we have the \defemph{diagonal $\cat{Q}$-matrix} $I_A\colon A\pto A$ defined as $I_A(a,a)=I_\cat{Q}$ and $I_A(a,a')=\bot_\cat{Q}$ (the least element of $\cat{Q}$) whenever $a\neq a'$.
Note that any $\cat{Q}$-category $\cat{C}$ can be seen as a $\cat{Q}$-matrix $\cat{C}\colon\ob{\cat{C}}\pto\ob{\cat{C}}$ satisfying $I_{\ob{\cat{C}}}\preceq \cat{C}$ and $\cat{C}\tms\cat{C}\preceq\cat{C}$.
We can view a presheaf $P$ over $\cat{C}$ as a $\cat{Q}$-matrix $P\colon \ob{\cat{C}}\pto 1$ satisfying $P\tms \cat{C}\preceq P$, where $1$ denotes a singleton. Dually, a copresheaf $R$ over $\cat{C}$ is a $\cat{Q}$-matrix $R\colon 1\pto \ob{\cat{C}}$ satisfying $\cat{C}\tms R\preceq R$.
The notations for the $\cat{Q}$-functors \eqref{eqn:Isbell_adj} constituting the Isbell adjunction agree with those for the above operations on $\cat{Q}$-matrices.

\begin{proposition}\label{prop:Isbell_hull_alternatively}
	For a $\cat{Q}$-category $\cat{C}$, define the set 
	\[
	\UA_\cat{C}=\{\,(X,Y)\mid X\colon \ob{\cat{C}}\pto 1,Y\colon 1\pto \ob{\cat{C}} \text{ and }Y\tms X \preceq \cat{C}\,\}.
	\]
	A pair $(X,Y)$ of $\cat{Q}$-matrices  $X\colon \ob{\cat{C}}\pto 1$ and $Y\colon 1\pto \ob{\cat{C}}$ belongs to $\Isb\cat{C}$
	if and only if it is maximal in $\UA_\cat{C}$,
	in the sense that (i)~$(X,Y)\in\UA_\cat{C}$, and (ii)~for any pair $(X',Y')\in\UA_\cat{C}$, if $X\preceq X'$
	and $Y\preceq Y'$ then $X=X'$ and $Y=Y'$.
\end{proposition}
\begin{proof}
	Suppose $(P,R)\in \Isb\cat{C}$.
	Then $R\tms P=(\cat{C}\mot P)\tms P\preceq \cat{C}$,
	so $(P,R)\in\UA_\cat{C}$. 
	Given any $(X',Y')\in\UA_\cat{C}$, 
	$P\preceq X'$ implies $Y'\preceq \cat{C}\mot X'\preceq \cat{C}\mot P=R$;
	similarly, $R\preceq Y'$ implies $X'\preceq P$.
	
	Conversely, suppose that $(X,Y)$ satisfies conditions (i) and (ii).
	Then the pair $(X,\cat{C}\mot X)$ satisfies 
	$(\cat{C}\mot X)\tms X\preceq \cat{C}$, $X\preceq X$ and, by (i), $Y\preceq \cat{C}\mot X$. So by (ii) we conclude $Y=\cat{C}\mot X$.
	Similarly, using the pair $(Y\mto \cat{C},Y)$ we see that $X=Y\mto \cat{C}$.
	In order to show that $X$ is a presheaf over $\cat{C}$, it suffices to show $X\tms \cat{C}\preceq X$. Since $X=Y\mto \cat{C}$, it suffices to show $X\tms \cat{C}\preceq Y\mto\cat{C}$, which is equivalent to $Y\tms X\tms \cat{C}\preceq\cat{C}$. Using $Y\tms X\preceq\cat{C}$ and $\cat{C}\tms\cat{C}\preceq\cat{C}$, we obtain the desired result. Similarly, $Y$ is a copresheaf over $\cat{C}$.
\end{proof}

\begin{proposition}
	Let $\cat{C}$ be a $\cat{Q}$-category. For each $(X,Y)\in\UA_\cat{C}$, there exists $(P,R)\in\MN\cat{C}$ such that $X\preceq P$ and $Y\preceq R$.
\end{proposition}
\begin{proof}
	This is immediate from Zorn's lemma, but a more explicit proof is also possible.
	
	We claim that $(P,R)=((\cat{C}\mot X)\mto\cat{C},\cat{C}\mot X)$ has the desired properties. 
	
	First, we have $X\preceq P$ and $Y\preceq R$, since the former is equivalent to $(\cat{C}\mot X)\tms X\preceq \cat{C}$, which in turn is equivalent to $\cat{C}\mot X\preceq \cat{C}\mot X$, whereas the latter is equivalent to $Y\tms X\preceq \cat{C}$.
	
	We show $(P,R)\in\MN\cat{C}$ using Proposition \ref{prop:Isbell_hull_alternatively}. $(P,R)$ is in $\UA_\cat{C}$ because $R\tms P=(\cat{C}\mot X)\tms ((\cat{C}\mot X)\mto\cat{C})\preceq \cat{C}$ is equivalent to  $(\cat{C}\mot X)\mto\cat{C}\preceq (\cat{C}\mot X)\mto\cat{C}$.
	To show $(P,R)$ is a maximal element in $\UA_\cat{C}$, suppose we are given any $(X',Y')\in \UA_\cat{C}$ with $P\preceq X'$ and $R\preceq Y'$. $R\preceq Y'$ implies $X'\preceq Y'\mto \cat{C}\preceq R\mto \cat{C}=P$; so we have $P= X'$. Using $X\preceq P=X'$, we have $Y'\tms X\preceq Y'\tms X'\preceq \cat{C}$, which is equivalent to $Y'\preceq \cat{C}\mot X=R$. So we also have $R=Y'$.
\end{proof}

\begin{remark}
	Jawhari, Misane and Pouzet define an analogue of the MacNeille completion for \emph{symmetric}\footnote{A $\cat{Q}$-category $\cat{C}$ over a commutative $\cat{Q}$ is \emph{symmetric} if $\cat{C}(c,c')=\cat{C}(c',c)$ for all $c,c'\in\cat{C}$.} $\cat{Q}$-categories over a \emph{commutative} (or more generally \emph{involutive}) \emph{integral}\footnote{A quantale $\cat{Q}$ is \emph{integral} if the unit $I_\cat{Q}$ is the greatest element in $({Q},\preceq_\cat{Q})$.} quantale $\cat{Q}$, and show that it is the injective envelope \cite{Jawhari_Misane_Pouzet_retracts}. Their definition is a variant of the alternative description of the MacNeille completion given in Proposition \ref{prop:Isbell_hull_alternatively}. As mentioned in Section \ref{sec:intro}, their result generalises Theorem \ref{thm:AP}.
\end{remark}

We can characterise complete $\cat{Q}$-categories in terms of the canonical embedding to the MacNeille completion.

\begin{proposition}
	A $\cat{Q}$-category $\cat{C}$ is complete if and only if the embedding $i_\cat{C}\colon\cat{C}\longrightarrow\MN\cat{C}$ is surjective (as a function between the sets of objects).
\end{proposition}
\begin{proof}
	Since $i_\cat{C}$ is always fully faithful and $\MN\cat{C}$ skeletal, $i_\cat{C}$ is surjective if and only if it is an equivalence of $\cat{Q}$-categories \cite[Proposition 4.4]{Stubbe_quantaloid_dist}. Since completeness is invariant under equivalence, if $i_\cat{C}$ is surjective then $\cat{C}$ is complete. For the converse, see e.g., \cite[Proposition 7.6]{Garner_topological}.
\end{proof}

\begin{corollary}\label{cor:completeness_alternatively}
	A $\cat{Q}$-category $\cat{C}$ is complete if and only if, given any $(X,Y)\in\UA_\cat{C}$, there exists $c\in\cat{C}$ such that $X\preceq \cat{C}(-,c)$ and $Y\preceq \cat{C}(c,-)$.
\end{corollary}

Now let us specialise to the case $\cat{Q}=\nonnegRmin$.
Extending the notion of Isbell convexity for di-spaces slightly, let us say a Lawvere metric space (= $\nonnegRmin$-category) $\cat{C}$ is \defemph{Isbell convex} if, for any family $((c_i,x_i, y_i))_{i\in I}$ where $c_i\in \cat{C}$ and $ x_i,y_i\in [0,\infty]$, if $x_i +y_j\geq \cat{C}(c_i,c_j)$ holds for each $i,j\in I$, then there exists $c\in\cat{C}$ such that $x_i\geq \cat{C}(c_i,c)$ and $y_i\geq \cat{C}(c,c_i)$ for all $i\in I$. 

\begin{proposition}
	A Lawvere metric space is Isbell convex if and only if it is complete (in the sense of Definition~\ref{def:completeness}).
\end{proposition}
The above proposition is immediate from Corollary \ref{cor:completeness_alternatively}.
So, modulo the above discussion (and the difference between Lawvere metric spaces and di-spaces), our Theorem~\ref{thm:main} yields Theorem~\ref{thm:KKO} when $\cat{Q}=\nonnegRmin$.
We remark that the proof of Theorem \ref{thm:KKO} in \cite{Kemajou_Kunzi_Otafudu_Isbell_hull}, however, relies heavily on Zorn's lemma and is quite different from our proof.

\medskip

A completely parallel comment applies to the relationship of our theorem when $\cat{Q}=\nonnegRminmax$ and K{\"u}nzi and Otafudu's characterisation of injective $\nonnegRminmax$-categories by \emph{q-spherical completeness} in \cite[Theorem~2]{Kunzi_Otafudu_ultra}.

\bibliographystyle{plain} %
\bibliography{myref} %
\end{document}